\documentclass[11pt]{article}
\usepackage{amsmath}
\usepackage{amssymb}
\usepackage{amsthm}
\usepackage[mathscr]{eucal}
\theoremstyle{definition}
\newtheorem{definition}{Definition}
\newtheorem{theorem}[definition]{Theorem}
\newtheorem{proposition}[definition]{Proposition}
\newtheorem{lemma}[definition]{Lemma}

\theoremstyle{remark}
\newtheorem{remark}[definition]{Remark}

\newcounter{enumctr}

\newcommand{\R}{\mathbb{R}}

\newcommand{\C}{\mathbb{C}}
\newcommand{\id}{\hbox{id}}

\renewcommand{\phi}{\varphi}

\newcommand{\rT}{\mathrm {T}}
\begin{document}
\title{\vspace*{-10mm}
On stable manifolds for fractional differential equations in high dimensional spaces}
\author{
N.D.~Cong\footnote{\tt ndcong@math.ac.vn, \rm Institute of Mathematics, Vietnam Academy of Science and Technology, 18 Hoang Quoc Viet, 10307 Ha Noi, Viet Nam},
T.S.~Doan\footnote{\tt dtson@math.ac.vn, \rm Institute of Mathematics, Vietnam Academy of Science and Technology, 18 Hoang Quoc Viet, 10307 Ha Noi, Viet Nam and Department of Mathematics, Hokkaido University, Japan},
S.~Siegmund\footnote{\tt stefan.siegmund@tu-dresden.de, \rm Center for Dynamics, Department of Mathematics, Technische Universit\"{a}t Dresden, Zellescher Weg 12-14, 01069 Dresden, Germany}
\;and\;
H.T.~Tuan\footnote{\tt httuan@math.ac.vn, \rm Institute of Mathematics, Vietnam Academy of Science and Technology, 18 Hoang Quoc Viet, 10307 Ha Noi, Viet Nam}}
\date{}
\maketitle
\begin{abstract}
Our aim in this paper is to establish stable manifolds near hyperbolic equilibria of fractional differential equations in arbitrary finite dimensional spaces.
\end{abstract}
\section{Introduction}
In recent years, fractional differential equations have attracted increasing interest due to
the fact that many mathematical problems in science and engineering can be modeled by fractional differential equations, see e.g.\ \cite{Kai,Kilbas,Podlubny}.

One of the biggest challenges in fractional differential equations is to understand the asymptotic behavior of fractional differential equations.  Although, there have been some publications on this topic such as stability theory, see e.g.\ \cite{Deng,LiChen} and \cite{LiZhang} for a survey of this topic, linear theory \cite{Bonilla, Odibat}, Lyapunov exponents \cite{LiGong,Cong_FCAA,CongSonSiegmundTuan} and linearization theory \cite{LiMa}, the development of a qualitative theory for fractional differential equations is still in its infancy. One of the reasons for this fact might be that general nonlinear fractional differential equations do not generate semigroups and the well-developed qualitative theory for dynamical systems cannot be applied directly.

The first step towards developing a stable manifold theorem for hyperbolic equilibria of fractional differential equations was done in \cite{Cong_2014}. In that paper, the authors consider a planar differential equation having zero as a steady state and the linear part of the linearization along the zero solution has one negative and one positive eigenvalue.  For this class of systems, an adequate Lyapunov--Perron operator is constructed and the graph of the stable manifold is characterized as the fixed point of this operator.

Our aim in this paper is to extend the result in \cite{Cong_2014} to fractional differential equations in an arbitrary finite dimensional space. The main difficulty arosen in this work is to introduce a suitable Lyapunov--Perron operator and show the contraction of this operator in an appropriate sense. To overcome this challenge, we first transform the linear part into a matrix which is very close to the diagonal one. Taking this approximated diagonal matrix as the linear part of the new system, we can use the variation of constants formula to define a Lyapunov--Perron operator. By this indirect way of construction, the Lyapunov--Perron operator can be estimated and enables us to prove the existence of the stable manifold.

A.\ Deshpande and V.\ Daftardar-Gejji~\cite{Deshpande} recently published a paper in which they formulated a stable manifold theorem similar to the one in this paper and the earlier reference \cite{CongSonSiegmundTuan1}. However, the proof of the stable manifold theorem in \cite{Deshpande} is incorrect and there are serious flaws in their approach to the problem of stable manifolds of fractional systems which we explain in detail and by counter-example in this paper.

The paper is organized as follows: in Section~\ref{sec.priliminaries}, we recall some fundamental results on fractional calculus and fractional differential equations. Section~\ref{sec.stab.mani} is devoted to the main result of this paper about stable manifolds for fractional differential equations in high dimensional spaces. Section~\ref{sec.discusstion} is devoted to a discussion about the work of A.\ Deshpande and V.\ Daftardar-Gejji~\cite{Deshpande} mentioned above.

To conclude this introductory section, we introduce notation which is used throughout this paper. Let $\R^d$ be endowed with the max norm, i.e.\ $\|x\|=\max(|x_1|,\dots,|x_d|)$ for all $x=(x_1,\dots,x_d)^{\rT}\in\R^d$. For $r>0$, let $B_{\R^d}(0,r):=\{x\in\R^d: \|x\|\leq r\}$ and for a Lipschitz continuous function $f:\R^d\rightarrow \R^d$, we define
\[
\ell_f(r):=\sup_{x,y\in B_{\R^d}(0,r)}\frac{\|f(x)-f(y)\|}{\|x-y\|}.
\]
Let $\R_{\geq 0}$ denote the set of all nonnegative real numbers and $\left(C_\infty(\R^d),\|\cdot\|_\infty\right)$ denote the space of all continuous functions $\xi:\R_{\geq 0}\rightarrow \R^d$ such that
\[
\|\xi\|_\infty:=\sup_{t\in \R_{\geq 0}}\|\xi(t)\|<\infty.
\]
It is well known that $\left(C_\infty(\R^d),\|\cdot\|_\infty\right)$ is a Banach space.
\section{Preliminaries}\label{sec.priliminaries}
\subsection{Fractional diferential equations}

%
We start this section by briefly recalling an abstract framework of fractional calculus and the corresponding fractional differential equations. We refer the reader to the books \cite{Kai,Kilbas} for more details about the theory of fractional differential equations.


Let $\alpha>0$ and $[a,b]\subset \R$. Let $f:[a,b]\rightarrow \R$ be a measurable function such that $f\in L_1([a,b])$, i.e.\ $\int_a^b|f(s)|\;ds<\infty$. Then, the \emph{Riemann-Liouville integral operator of order $\alpha$} is defined by
\[
(I_{a+}^{\alpha}f)(x):=\frac{1}{\Gamma(\alpha)}\int_a^x(x-t)^{\alpha-1}f(t)\;dt\quad \hbox{ for } x>a,
\]
where the \emph{Euler Gamma function } $\Gamma:(0,\infty)\rightarrow \R$ is defined as
\[
\Gamma(\alpha):=\int_0^\infty t^{\alpha-1}\exp(-t)\;dt,
\]
see e.g.,\ \cite{Kai}. The corresponding \emph{Riemann-Liouville fractional derivative} is given by
\[
(D_{a+}^\alpha f)(x):=(D^mI_{a+}^{m-\alpha}f)(x),
\]
where $D=\frac{d}{dx}$ is the usual derivative and $m:=\lceil\alpha\rceil$ is the smallest integer larger or equal $\alpha$. On the other hand, the \emph{Caputo fractional derivative } $^{C}D_{a+}^\alpha f$ of a function $f\in C^m([a,b])$, which was introduced by Caputo (see e.g.,\ \cite{Kai}), is defined by
\[
(^{C}D_{a+}^\alpha f)(x):=(I_{a+}^{m-\alpha}D^mf)(x),\qquad \hbox{ for } x>a.
\]
The Caputo fractional derivative of a $d$-dimensional vector function $f(t)=(f_1(t),\dots,f_d(t))^{\rT}$ is defined component-wise as $$^{C}D^\alpha_{0+}f(t)=(^{C}D^\alpha_{0+}f_1(t),\dots,^{C}D^\alpha_{0+}f_d(t))^{\rT}.$$
We refer the reader to \cite[Chapters 2 and 3]{Kai} for a discussion on some advantages of the Caputo derivative in comparison to the Riemann--Liouville derivative. In this paper, we consider nonlinear fractional differential equations involving the Caputo fractional derivative of the following form
\begin{equation}\label{Eq1a}
^{C}D_{0+}^\alpha x(t)=Ax(t)+f(x(t)),
\end{equation}
where $\alpha\in (0,1)$, $A\in\R^{d\times d}$ and $f:\R^d\rightarrow \R^d$ is continuous on $\R^d$. Assume that for any initial value $x\in\R^d$, the initial value problem \eqref{Eq1} has a unique solution denoted by $\phi(\cdot,x)$ which is defined on the whole interval $\R_{\geq 0}$. This is for instance the case if $f$ is globally Lipschitz, see e.g., \cite[Theorem 3.1]{Lin}.

For $f=0$, system \eqref{Eq1a} reduces to a linear time-invariant fractional differential equation
\begin{equation}\label{Eq2}
^{C}D_{0+}^\alpha x(t)=Ax(t).
\end{equation}
As shown in \cite{Bonilla}, $E_{\alpha}(t^{\alpha}A)x$ solves \eqref{Eq2} with the initial condition $x(0)=x$, where the \emph{Mittag-Leffler matrix function} $E_{\alpha}(A)$ for a matrix $A\in\R^{d\times d}$ is defined as
\[
E_{\alpha,\beta}(A):=\sum_{k=0}^\infty \frac{A^k}{\Gamma(\alpha k+\beta)},\qquad E_{\alpha}(A):=E_{\alpha,1}(A),
\]
where $\beta \in \C$.

If the nonlinear term $f$ does not vanish, it is in general impossible to provide an explicit form of the solution of \eqref{Eq1a}. However, we get a presentation of solutions for \eqref{Eq1a} by using the Mittag-Leffler matrix function. We refer the reader to \cite{Bonilla,Kilbas,LiPeng} for a proof of this result.
\begin{theorem}[Variation of constants formula for fractional differential equations]\label{Var_Const_Form}
The solution $\phi(\cdot,x)$ of \eqref{Eq1a}, $x(0) = x$, satisfies for $t\in\R_{\geq 0}$
\begin{equation}\label{Eq4}
\phi(t,x)=E_{\alpha}(t^{\alpha}A)x+\int_0^t (t-s)^{\alpha-1}E_{\alpha,\alpha}((t-s)^{\alpha}A)f(\phi(s,x))\;ds.
\end{equation}
\end{theorem}
\subsection{Mittag-Leffler function}\label{MF_Function}
This subsection is devoted to an introduction of the so-called Mittag-Leffler function and its basic properties. Analog to the exponential function frequently used in the investigation of integer-order systems, this function naturally occurs as the solution of fractional differential equations. For $\alpha \in (0,1)$, $\beta \in \C$ the \emph{Mittag-Leffler function} $E_{\alpha,\beta}:\C\rightarrow \C$ is defined as
\[
E_{\alpha,\beta}(z):=\sum_{k=0}^\infty \frac{z^k}{\Gamma(\alpha k+\beta)},\qquad E_{\alpha}(z):=E_{\alpha,1}(z).
\]
For a complex number $z$ we denote by $\text{arg}(z)$ its argument, which we define to take value in the interval $(-\pi,\pi]$, i.e., $-\pi<\text{arg}(z)\leq \pi$ for all $z\in\C\setminus\{0\}$.
\begin{lemma}[Asymptotic expansion of Mittag-Leffler function]\label{lemma1}
Let $0<\alpha<1$, $\beta \in \C$ and a real number $\mu$ such that $\frac{\alpha \pi}{2}<\mu <\alpha \pi$.
\begin{itemize}
\item [(i)] If $|\text{arg}(z)|\leq \mu$, then for an arbitrary integer $p\geq 1$ the following expansion holds as $|z|\to \infty$
\[
E_{\alpha,\beta}(z)= \frac{1}{\alpha}z^{\frac{1-\beta}{\alpha}}\exp{(z^{\frac{1}{\alpha}})}-\sum_{k=1}^p\frac{z^{-k}}{\Gamma(\beta-\alpha k)}+\mathcal{O}(|z|^{-1-p}).
\]
\item [(ii)] If $\mu \leq |\text{arg}(z)|\leq \pi$, then for an arbitrary integer $p\geq 1$ the following expansion holds as $|z|\to \infty$
\[
E_{\alpha,\beta}(z)= -\sum_{k=1}^p \frac{z^{-k}}{\Gamma(\beta-\alpha k)}+\mathcal{O}(|z|^{-1-p}).
\]
\end{itemize}
\end{lemma}
\begin{proof}
See \cite[Theorem 1.3, p.\ 32]{Podlubny} and \cite[Theorem 1.4, pp.~33--34]{Podlubny}.
\end{proof}
\section{Stable manifolds for fractional differential equations}\label{sec.stab.mani}
%
Consider nonlinear fractional differential equations of the following form
\begin{equation}\label{Eq1}
^{C}D_{0+}^{\alpha}x(t)=Ax(t)+f(x(t)),
\end{equation}
where $A\in \R^{d\times d}$ and $f:\R^d\rightarrow \R^d$ is continuous on $\R^d$ and Lipschitz continuous in a neighborhood of the origin satisfying \begin{equation}\label{Lipschitz condition}
f(0)=0\quad\hbox{and}\quad\lim_{r\to 0} \ell_f(r)=0.
\end{equation}
Note that $f$ fulfills condition \eqref{Lipschitz condition} provided that $f$ is $C^1$ in a neighborhood of the origin with $f(0)=0$ and $Df(0)=0$. Furthermore, we suppose the following hyperbolicity assumption on $A$ as follows.

\textbf{Hyperbolicity}: Let $\{\hat\lambda_1,\dots,\hat\lambda_m\}$ be the collection of all distinct (complex) eigenvalues of $A$. We say that $A$ satisfies the hyperbolicity assumption if 
\begin{equation}\label{Ordering}
|\hbox{arg}(\hat\lambda_i)| \not= \frac{\alpha\pi}{2} \quad\hbox{for all}\quad i=1,\ldots,m.
\end{equation}

We recall the following definition of local stable manifold for fractional differential equations, cf.\ \cite[Definition 4]{Cong_2014}.
\begin{definition} Let $U\subseteq \R^d$ be an arbitrary neighborhood of $0$. The \emph{stable manifold} of \eqref{Eq1} in $U$ is defined as
\[
\mathcal W^s(U):=\Big\{x\in U: \phi(t,x)\in U \hbox{ for } t\in\R_{\geq 0}\hbox{ and } \lim_{t\to\infty} \phi(t,x)=0\Big\}.
\]
\end{definition}
Before going into the details of the construction of the stable manifold, we briefly discuss the structure of this section:
As in \cite{Cong_2014}, we plan to characterize the graph of the stable manifold of \eqref{Eq1} by the fixed point of a suitable Lyapunov--Perron operator. In light of the variation of constants formula in Theorem \ref{Var_Const_Form}, we are able to construct a Lyapunov--Perron operator associated with \eqref{Eq1} involving the Mittag-Leffler matrix functions $E_{\alpha}(t^\alpha A)$. More precisely, let $A=\hbox{diag}(A^u,A^s)$, where $A^u\in\R^{d^u\times d^u}$ and $A^s\in\R^{d^s\times d^s}$ are the unstable and stable part of $A$, respectively. Then, for each $x^s\in \R^{d^s}$ the Lyapunov--Perron operator $\mathcal{LP}_{x^s}:C_{\infty}(\R^d)\rightarrow C_{\infty}(\R^d)$ is defined by $\mathcal{LP}_{x^s} (\xi)(t)=((\mathcal{LP}_{x^s} \xi)^u(t),(\mathcal{LP}_{x^s} \xi)^s(t))$ for $t\in\R_{\geq 0}$, where
\begin{eqnarray*}
(\mathcal{LP}_{x^s}\xi)^u(t)
&:=&
\int_0^t (t-\tau)^{\alpha-1}E_{\alpha,\alpha}((t-\tau)^\alpha A^u)f^u(\xi(\tau))\;d\tau -\\
&&
\hspace{-25mm}E_{\alpha}(t^\alpha A^u)\lim_{s\to\infty}(E_{\alpha}(s^\alpha A^u))^{-1}\int_0^s(s-\tau)^{\alpha-1}E_{\alpha,\alpha}((s-\tau)^\alpha A^u)f^u(\xi(\tau))\;d\tau,
\\[1.5ex]
(\mathcal{LP}_{x^s}\xi)^s(t)
&:=&
E_{\alpha}(t^{\alpha} A^s) x^s+\int_0^t (t-\tau)^{\alpha-1} E_{\alpha,\alpha}((t-\tau)^{\alpha}A^s)f^s(\xi(\tau))\;d\tau.
\end{eqnarray*}

However, for such a Lyapunov--Perron operator, it is very difficult to compute explicitly the limit
\begin{equation}\label{NewEquation_01}
\lim_{s\to\infty}(E_{\alpha}(s^\alpha A^u))^{-1}\int_0^s (s-\tau)^{\alpha-1}E_{\alpha,\alpha}((s-\tau)^\alpha A^u)f^u(\xi(\tau))\;d\tau
\end{equation}
and to show the contraction property of this operator because the estimates of the matrix valued-functions $E_\alpha (t^\alpha A)$ and $E_{\alpha,\alpha} (t^\alpha A)$ are complicated, see e.g., \cite{Kai}. These estimates become much easier when $A^u$ and $A^s$ are diagonal, see \cite{Cong_2014}. So, our idea is to transform $A^u$ and $A^s$ to new matrices being close to diagonal matrices and we use these diagonal matrices as the linear part. In summary, our construction of a Lyapunov--Perron operator can be done as follows:
\begin{itemize}
\item Step 1: We first transform the matrix associated with the linear system to a new matrix being different from a diagonal matrix by a small nilpotent matrix. The correspondence between the stable manifold of the original system and the new system is established in Lemma \ref{LinearTransformation}.
\item Step 2: Interpreting the small nilpotent part constructed in Step 1 as a term of order $2$, we obtain a new system whose linear part is of diagonal form. Using the variation of constants formula, we construct a Lyapunov-Perron operator of the new system. For such an operator, a contraction property is proved in Proposition \ref{Prp1}.
\end{itemize}
After showing the contraction of a suitable Lyapunov operator, using an analogous argument as in \cite{Cong_2014}, we obtain the existence and properties of the stable manifold of \eqref{Eq1}:
\begin{theorem}[Stable manifold theorem for fractional differential equations]\label{Theorem1}
There exists $r>0$ such that the stable manifold $\mathcal W^s( B_{\R^{d}}(0,r))$ of \eqref{Eq1} is given as the graph of a Lipschitz function $g:B_{\R^{d^s}}(0,r)\rightarrow B_{\R^{d^u}}(0,r)$ which satisfies the following properties:
\begin{itemize}
\item [(i)] $g(0)=0$.
\item [(ii)] The map $g$ is Lipschitz continuous.
\end{itemize}
\end{theorem}
\subsubsection*{Transformation of the linear part}
Using \cite[Theorem 6.37, pp.~146]{Shilov}, there exists a nonsingular matrix $T\in\C^{d\times d}$ transforming $A$ into the Jordan normal form, i.e.
\[
T^{-1}A T=\hbox{diag}(A_1,\dots,A_n),
\]
where for $i=1,\dots,n$ the block $A_i$ is of the following form
\[
A_i=\lambda_i\, \id_{d_i\times d_i}+\eta_i\, N_{d_i\times d_i},
\]
where $\eta_i\in\{0,1\}$, $\lambda_i \in \{\hat\lambda_1,\ldots,\hat\lambda_m\}$, and the nilpotent matrix $N_{d_i\times d_i}$ is given by
\[
N_{d_i\times d_i}:=
\left(
      \begin{array}{*7{c}}
      0  &     1         &    0      & \cdots        &  0        \\
        0        & 0    &    1     &   \cdots      &              0\\
        \vdots &\vdots        &  \ddots         &          \ddots &\vdots\\
        0 &    0           &\cdots           &  0 &          1 \\

        0& 0  &\cdots                                          &0         & 0 \\
      \end{array}
    \right)_{d_i \times d_i}.
\]
Let us notice that by this transformation we go from the field of real numbers  out to the field of complex numbers, and we may remain in the field of real numbers only if all eigenvalues of $A$ are real. For a general real-valued matrix $A$ we may simply embed $\R$ into $\C$, consider $A$ as a complex-valued matrix and thus get the above Jordan form for $A$. Alternatively, we may use a more cumbersome real-valued Jordan form (for discussion of a similar issue for FDE see also Diethelm~\cite[pp.\ 152--153]{Kai}). For simplicity we use the embedding method and omit the discussion on how to return back to the field of real numbers. Note also that this kind of technique is well known in the theory of ordinary differential equations. 

 By the hyperbolicity assumption \eqref{Ordering} we may order $\lambda_i$, $i=1,\ldots,n$, in such a way that there exists an index $k\in \{1,\ldots, n-1\}$ with the property that
$$
|\hbox{arg}(\lambda_1)|,\dots,|\hbox{arg}(\lambda_k)|<\frac{\alpha\pi}{2}<|\hbox{arg}(\lambda_{k+1})|,\dots,|\hbox{arg}(\lambda_n)|.
$$
Let $\delta$ be an arbitrary but fixed positive number. Using the transformation $P_i:=\textup{diag}(1,\delta,\dots,\delta^{d_i-1})$, we obtain that
\begin{equation*}
P_i^{-1} A_i P_i=\lambda_i\, \id_{d_i\times d_i}+\delta_i\, N_{d_i\times d_i},
\end{equation*}
$\delta_i\in \{0,\delta\}$. Hence, under the transformation $y:=(TP)^{-1}x$ system \eqref{Eq1} becomes
\begin{equation}\label{NewSystem}
^{C}D_{0+}^\alpha x(t)=\hbox{diag}(J_1,\dots,J_n)x(t)+h(x(t)),
\end{equation}
where $J_i:=\lambda_i \id_{d_i\times d_i}$ for $i=1,\dots,n$ and the function $h$ is given by
\begin{equation}\label{Eq3}
h(x):=\text{diag}(\delta_1N_{d_1\times d_1},\dots,\delta_nN_{d_n\times d_n})x+(TP)^{-1}f(TPx).
\end{equation}
\begin{remark}\label{Remark1}
Note that the map $x\mapsto \text{diag}(\delta_1N_{d_1\times d_1},\dots,\delta_nN_{d_n\times d_n})x$ is a Lipschitz continuous function with a Lipschitz constant $\delta$. Thus, by \eqref{Lipschitz condition} we have
\[
h(0)=0,\qquad \lim_{r\to 0}\ell_h(r)\le \delta.
\]
\end{remark}

Next, we establish the relation between the stable manifold of the new system \eqref{NewSystem} and the stable manifold of the original system \eqref{Eq1}.
\begin{lemma}\label{LinearTransformation}
For any neighborhood $U\subseteq \R^d$ of the origin, let $\mathcal{W}^s_{\rm new}(U)$ denote the stable manifold of \eqref{NewSystem} in $U$. Then,
\[
\mathcal{W}^s(U)=TP\mathcal{W}^s_{\rm new}((TP)^{-1}U),
\]
where $\mathcal{W}^s(U)$ is the stable manifold of \eqref{Eq1} in $U$.
\end{lemma}
\begin{proof}
Let $\phi_{\rm new}(\cdot,x)$ denote the solution of \eqref{NewSystem} satisfying $\phi_{\rm new}(0,x)=x$. Then,
\[
\phi(t,x) := (TP)\phi_{\rm new}(t,(TP)^{-1}x) \qquad\hbox{for all } (t,x)\in\R_{\geq 0}\times \R^d
\]
is the solution of \eqref{Eq1} satisfying $\phi(0,x)=x$.
Let $U$ be an arbitrary neighborhood of the origin. Recall that $x\in {\mathcal W}^s(U)$ if and only if $\phi(t,x)\in U$ for all $t\in\R_{\geq 0}$ and $\lim_{t\to\infty}\phi(t,x)=0$. Equivalently,
\[
\phi_{\rm new}(t,(TP)^{-1}x)\in (TP)^{-1}(U)\quad \hbox{and}\quad \lim_{t\to\infty}\phi_{\rm new}(t,(TP)^{-1}x)=0,
\]
which is again equivalent to $(TP)^{-1}x\in \mathcal{W}^s_{\rm new}((TP)^{-1}U)$. Consequently, $\mathcal{W}^s(U)=TP\mathcal{W}^s_{\rm new}((TP)^{-1}U)$ and the proof is complete.
\end{proof}
\subsubsection*{Construction of an appropriate Lyapunov--Perron operator}
According to Lemma \ref{LinearTransformation}, the stable manifold of \eqref{Eq1} in a neighborhood $U$ can be represented by the stable manifold in the neighborhood $(TP)^{-1}U$ of \eqref{NewSystem}. Hence, our focus now is to investigate the stable manifold of \eqref{NewSystem}. Analog to the approach used in \cite{Cong_2014}, we will construct the graph of the stable manifold as the fixed point of the corresponding Lyapunov--Perron operator. Before doing this, we discuss some conventions which are used in the remaining part of this section: The space $\R^d$ can be written as $\R^d=\R^{d_1}\times\dots\times\R^{d_n}$ or $\R^d=\R^{d^u}\times \R^{d^s}$ where $d^u:=d_1+\dots+d_k$ and $d^s:=d_{k+1}+\dots+d_n$. A vector $x\in\R^d$ can be written component-wise as $x=(x^u,x^s)$ or $x=(x^1,\dots,x^n)$, respectively.

%
%

%
For any $x^s=(x^{k+1},\dots,x^{n})^{\rm T}\in \R^{d^s}=\R^{d_{k+1}}\times\dots\times \R^{d_n}$, the operator $\mathcal{T}_{x^s}: C_\infty(\R^d)\rightarrow C_\infty(\R^d)$ is defined by
\[
\mathcal{T}_{x^s}\xi(t)=((\mathcal{T}_{x^s}\xi)^1(t),\dots,(\mathcal{T}_{x^s}\xi)^n(t))^{\rT}\qquad\hbox{for } t\in\R_{\geq 0},
\]
where for $i=1,\dots,k$
\begin{multline*}
(\mathcal{T}_{x^s}\xi)^i(t)=\int_0^t (t-\tau)^{\alpha-1}E_{\alpha,\alpha}((t-\tau)^\alpha J_i)h^i(\xi(\tau))\;d\tau\\
 -\lambda_i^{\frac{1}{\alpha}-1}E_\alpha(t^\alpha J_i)\int_0^\infty \exp{(-\lambda_i^{\frac{1}{\alpha}}\tau)}h^i(\xi(\tau))\;d\tau,
\end{multline*}
and for $i=k+1,\dots,n$
\[
(\mathcal{T}_{x^s}\xi)^i(t)
=
E_\alpha(t^\alpha J_i)x^i+\int_0^t (t-\tau)^{\alpha-1}E_{\alpha,\alpha}((t-\tau)^\alpha J_i)h^i(\xi(\tau))\;d\tau,
\]
is called the \emph{Lyapunov-Perron operator associated with \eqref{NewSystem}}.

An important step towards establishing the stable manifold is to estimate the Lyapunov--Perron operator. In \cite{Cong_2014}, an estimate on the Lyapunov--Perron operator for $n=2$ and $d_1=d_2=1$ was developed. This estimate can be easily extended to our situation when all matrices $J_i$ are of diagonal form. Note that a minor change is due to the fact that the eigenvalues $\lambda_i$ can be complex. However, Lemma \ref{lemma1} in Subsection \ref{MF_Function} provides the asymptotic behavior of the Mittag-Leffler function $E_{\alpha}(z)$ also in this case. Thus, we arrive at the following estimate.
\begin{proposition}\label{Prp1} Consider system \eqref{NewSystem}. Then, there exists a constant $C(\alpha,\mathbf{\lambda})$ depending on $\alpha,\lambda:=(\lambda_1,\dots,\lambda_n)$ such that for all $x^s,\widehat x^s\in\R^{d^s}$ and $\xi,\widehat\xi\in C_\infty(\R^d)$ the following inequality holds
\begin{align}\label{Contraction}
\notag \|\mathcal T_{x^s}\xi-\mathcal T_{\widehat{x}^s}\widehat\xi\|_\infty
\leq &
\max_{k+1\le i \le n}\sup_{t\geq 0}|E_\alpha(\lambda_i t^\alpha)| \|x^s-\widehat{x}^s\|\\
&+C(\alpha,\lambda)\;
\ell_h(\max(\|\xi\|_\infty,\|\widehat\xi\|_\infty))\|\xi-\widehat\xi\|_\infty.
\end{align}
Consequently, $\mathcal T_{x^s}$ is well-defined and
\begin{equation}\label{well_defined}
\|\mathcal T_{x^s}\xi-\mathcal T_{x^s}\widehat\xi\|_\infty \leq C(\alpha,\lambda)\; \ell_h(\max(\|\xi\|_\infty,\|\widehat\xi\|_\infty))\|\xi-\widehat\xi\|_\infty.
\end{equation}
\end{proposition}
So far, we have proved that the Lyapunov--Perron operator is well-defined and Lipschitz continuous. Note that the Lipschitz constant $C(\alpha,\lambda)$ is independent of the constant $\delta$ which is hidden in the coefficients of system \eqref{NewSystem}. From now on, we choose and fix the constant $\delta$ as follows $\delta:=\frac{1}{3C(\alpha,\lambda)}$. The remaining difficult question is now to choose a ball with small radius in $C_\infty(\R^d)$ such that the restriction of the Lyapunov--Perron operator to this ball is strictly contractive. A positive answer to this question is given in the following technical lemma.
\begin{proposition}\label{Lemma6}
The following statements hold:
\begin{itemize}
\item [(i)] There exists $r^*>0$ such that
\begin{equation}\label{Eq7a}
C(\alpha,\lambda)\;  \ell_h(r^*) \leq \frac{2}{3}.
\end{equation}
\item [(ii)] Choose and fix $r^*>0$ satisfying \eqref{Eq7a}. Define
\begin{equation}\label{Eq7b}
r:=\frac{r^*}{3\max_{k+1\le i\le n}\sup_{t\ge 0}|E_\alpha(\lambda_it^\alpha)|}.
\end{equation}
Let  $B_{C^{\infty}}(0,r^*):=\{\xi\in C_\infty(\R^d):\left||\xi|\right|_\infty\le r^*\}$. Then, for any $x^s\in B_{\R^{d_s}}(0,r)$ we have $\mathcal T_{x^s} (B_{C^{\infty}}(0,r^*))\subset B_{C^{\infty}}(0,r^*)$ and
\begin{equation*}\label{LipschitzContinuity}
\|\mathcal T_{x^s}\xi-\mathcal T_{x^s}\widehat {\xi}\|_\infty
\leq
\frac{2}{3}\|\xi-\widehat{\xi}\|_\infty\quad\hbox{ for all } \xi,\widehat{\xi}\in B_{C^{\infty}}(0,r^*).
\end{equation*}
\end{itemize}

\end{proposition}
\begin{proof}
(i) By Remark \ref{Remark1}, $\lim_{r\to 0}\ell_h(r)\le \delta$. Since $\delta C(\alpha,\lambda)=\frac{1}{3}$, the assertion (i) is proved.

(ii) Let $x^s\in \R^{d^s}$ be arbitrary with $\|x^s\|\leq r$. Let $\xi\in B_{C^{\infty}}(0,r^*)$ be arbitrary. According to \eqref{Contraction} in Proposition \ref{Prp1}, we obtain that
\begin{eqnarray*}
\|\mathcal T_{x^s}\xi\|_\infty
&\leq&
 \max_{k+1\le i\le \ell}\sup_{t\ge 0}|E_\alpha(\lambda_it^\alpha)|\|x^s\|
 +
 C(\alpha,\lambda)\,\ell_{h}(r^*)\|\xi\|_{\infty}\\[1ex]
 &\leq&
 \frac{1}{3}r^*+\frac{2}{3}r^*,
\end{eqnarray*}
which proves that $\mathcal T_{x^s}(B_{C^{\infty}}(0,r^*))\subset B_{C^{\infty}}(0,r^*)$. Furthermore, by  Proposition \ref{Prp1} and part (i) for all $x^s\in B_{\R^{d^s}}(0,r)$ and $\xi,\widehat{\xi}\in B_{C^{\infty}}(0,r^*)$ we have
\begin{eqnarray*}
\|\mathcal T_{x^s}\xi-\mathcal T_{x^s}\widehat{\xi}\|_\infty
&\leq&
C(\alpha,\lambda)\ell_{h}(r^*)\;\|\xi-\widehat{\xi}\|_\infty\\[1.5ex]
&\leq &
\frac{2}{3}\|\xi-\widehat{\xi}\|_\infty,
\end{eqnarray*}
which concludes the proof.
\end{proof}
\subsubsection*{Construction and properties of stable manifolds}
Before we are able to show Theorem \ref{Theorem1}, we need the following result about a one-to-one correspondence between the graph of the stable manifold and the fixed point of the Lyapunov--Perron operator. In \cite{Cong_2014} a detailed proof of this result for planar systems is provided and it is straightforward to extend this proof to the higher dimensional case.
\begin{proposition}\label{Prp2} For each $x=(x^u,x^s)^{\rT}\in\R^d$, let $\phi_{\rm new}(\cdot,x)$ denote the solution of \eqref{NewSystem} with the initial condition $\phi_{\rm new}(0,x)=x$. Then, the following statements hold:
\begin{itemize}
\item [(i)] If $x\in \mathcal{W}_{\rm new}^s(U)$, then $\phi_{\rm new}(\cdot,x)$ is a fixed point of $\mathcal T_{x^s}$.
\item [(ii)] If $\xi(t)=(\xi^u(t),\xi^s(t))^{\rT}$ is a fixed point of $\mathcal T_{x^s}$ for some $x^s\in\R^{d^s}$, then $\xi(t)$ is a solution of \eqref{NewSystem} with $\xi^s(0)=x^s$. Furthermore, suppose additionally that $\|\xi\|_\infty\le r^*$, where
$r^*$ satisfies \eqref{Eq7a}. Then, $\lim_{t\to\infty}\xi(t)=0$.
\end{itemize}
\end{proposition}
\begin{proof}
Analog to \cite[Proposition 9]{Cong_2014}.
\end{proof}
The existence and properties of the stable manifold of \eqref{NewSystem} are established in the following theorem.
\begin{theorem}\label{TheoremDE}
Consider system \eqref{NewSystem}. Let $r^*$ satisfy \eqref{Eq7a} and $r$ be defined as in \eqref{Eq7b}.
Then, for any $x^s\in B_{\R^{d^s}}(0,r)$, there exists a unique $w(x^s)\in B_{\R^{d^u}}(0,r^*)$ such that $(w(x^s),x^s)\in \mathcal W^s_{\rm new}(B_{\R^d}(0,r^*))$. Furthermore, the map $w: B_{\R^{d^s}}(0,r)\rightarrow B_{\R^{d^u}}(0,r^*)$ satisfies the following properties:
\begin{itemize}
\item [(i)] $w(0)=0$.
\item [(ii)] The map $w$ is Lipschitz continuous. More precisely, for all $x^s,\widehat {x}^s\in B_{\R^{d^s}}(0,r)$ we have
\[
\|w(x^s)-w(\widehat{x}^s)\|\leq 3 \max_{k+1\le i\le n}\sup_{t\geq 0}|E_\alpha(\lambda_it^\alpha)|\|x^s-\widehat{x}^s\|.
\]
\end{itemize}
\end{theorem}
\begin{proof}
Let $x^s\in B_{\R^{d^s}}(0,r)$ be arbitrary. Using Proposition \ref{Lemma6} and the Contraction Mapping Principle, there exists a unique fixed point $\xi\in B_{C^{\infty}}(0,r^*)$ of $\mathcal T_{x^s}$. This, together with Proposition \ref{Prp2}, implies that there exists a unique element $x^u\in B_{\R^{d^u}}(0,r^*)$ such that $(x^u,x^s)^{\rT}\in\mathcal{W}_{\rm new}^s(B_{\R^d}(0,r^*))$. Define $w(x^s):=x^u$. To conclude the proof, the rest of this step is devoted to verify properties (i) and (ii) for the function $w$. Since $0$ is a fixed point of $\mathcal T_0$, it follows that $(0,0)^{\rT}\in \mathcal{W}_{\rm new}^s (B_{\R^d}(0,r^*))$. Therefore, $w(0)=0$ and the assertion (i) is thus proved. To prove (ii), let $x^s,\widehat{x}^s\in B_{\R^{d^s}}(0,r)$ be arbitrary. Define $\xi:=\phi_{\rm new}(\cdot,(w(x^s),x^s)^{\rT})$ and $\widehat{\xi}:=\phi_{\rm new}(\cdot,(w(\widehat{x}^s),\widehat{x}^s)^{\rT})$. Using Proposition \ref{Prp1} and Proposition \ref{Prp2}, we obtain that
\[
\|\xi-\widehat{\xi}\|_\infty=\|\mathcal T_{x^s}\xi-\mathcal T_{\widehat{x}^s}\widehat{\xi}\|_\infty
\leq
\|x^s-\widehat{x}^s\|\max_{k+1\le i\le n}\sup_{t\ge 0}|E_\alpha(\lambda_it^\alpha)|+\frac{2}{3}\|\xi-\widehat{\xi}\|_\infty,
\]
which implies that
\[
\|w(x^s)-w(\widehat{x}^s)\|
\leq \|\xi-\widehat{\xi}\|_\infty
\leq 3\max_{k+1\le i\le n}\sup_{ t\ge 0}|E_\alpha(\lambda_it^\alpha)|\|x^s-\widehat{x}^s\|,
\]
and the proof is complete.
\end{proof}
\begin{proof}[Proof of Theorem \ref{Theorem1}]
Using Theorem \ref{TheoremDE}, the stable manifold of \eqref{NewSystem} can be represented as the graph of a Lipschitz continuous function. This together with Lemma \ref{LinearTransformation} implies that the stable manifold of the original system \eqref{Eq1} is also given as the graph of a Lipschitz continuous function. The proof of this theorem is complete.
\end{proof}
\section{Discussion of the paper by A.\ Deshpande and V.\ Daftardar-Gejji~\cite{Deshpande}}\label{sec.discusstion}
In this section, we discuss the paper ~\cite{Deshpande} by A.\ Deshpande and V.\ Daftardar-Gejji. As we stated in the Introduction, A. Deshpande and V. Daftardar-Gejji~\cite{Deshpande} formulated a stable manifold theorem \cite[Theorem 3]{Deshpande}, but the proof they gave is incorrect. To show this, we give an example to illustrate the fact that the manifold they constructed in the proof of the theorem is not the stable manifold as they claimed. The example is a simple two-dimensional system where such a stable manifold is already known to exist, see Cong et al.~\cite{Cong_2014}. Namely, let us consider the two-dimensional fractional system 
\begin{equation}\label{mainsystem}
^{C}D_{0+}^\alpha x(t)=Ax(t)+f(x(t)),\qquad x(0)=\hat{x}\in \R^2,
\end{equation}
of order $\alpha\in (0,1)$, where  
\[
 A=   \left(
      \begin{array}{*5{c}}
       -2    &     0 \\
     0      & 2   \\
      \end{array}
    \right),
\] 
and $$f(x_1,x_2)=\left(
      \begin{array}{*5{c}}
     x_1^2\\
   x_1^2+x_2^2  \\
      \end{array}
    \right).$$
   From here to the end of this section we use the notation $N_r(0)$, $E^s$, $T_\sigma$, $\pi_s$, $\pi_u$, $B(t)$ for the system \eqref{mainsystem} as defined in the paper \cite{Deshpande}.
According to A. Deshpande and V. Daftardar-Gejji~\cite[Proof of Theorem 3, pp. 2448--2450]{Deshpande}, one can find a local stable manifold of \eqref{mainsystem} near the origin, 
denoted by $W^s_{loc}(N_r(0))$, as follows. For each $\sigma\in N_\delta(0)\cap E^s$ one finds a unique solution of \eqref{mainsystem}, say $\varphi(\cdot,\hat{x})$ with $\varphi(0,\hat{x})=\hat{x}\in\R^2$, which is a fixed point of an operator $T_\sigma$ such that $\pi_s(\hat{x})=\sigma$. Let ${\mathcal S}(\sigma) = \pi_u(\hat{x})$ then $(\sigma, {\mathcal S}(\sigma))=(\pi_s(\hat{x}),\pi_u(\hat{x}))$ will form a graph over $N_\delta(0)\cap E^s$ as $\sigma$ varies, and this graph is the stable manifold $W^s_{loc}(N_r(0))$. We show below that this claim is false for the system \eqref{mainsystem}.

Indeed, since $A$ is diagonal the stable space $E^s$ is easily found to be $E^s=\{(x,0)^{\rm T}: x\in \R\}$, and the projections $\pi_s,\pi_u$ are projections onto coordinate axes, and we can identify the images of the projections of a vector with the coordinates of the vectors.  The relation $\sigma = (\sigma_1,\sigma_2)^{\rT} \in N_\delta(0)\cap E^s$ is equivalent to $\sigma_2=0$ and $|\sigma_1|\leq \delta$.  For one such $\sigma\not= 0$ we construct the operator $T_\sigma$ and find the fixed point $\varphi(\cdot,\hat{x})$ of $T_\sigma$ according to \cite{Deshpande}. A claim of the stable manifold is that the function $\varphi(t,\hat{x})$ tends to 0 as $t\rightarrow\infty$ as the authors of \cite{Deshpande} tried to prove in Step 2 of the proof of their theorem. For the system  \eqref{mainsystem}  we have
\begin{align*}
\pi_u (T_\sigma \varphi(t,\hat{x})) &=-E_\alpha(2 t^\alpha)\int_0^\infty \alpha B(-\tau)[\varphi_1^2(\tau,\hat{x})+\varphi_2^2(\tau,\hat{x})]\;d\tau\\
&+\int_0^t(t-\tau)^{\alpha-1}E_{\alpha,\alpha}(2(t-\tau)^\alpha)[\varphi_1^2(\tau,\hat{x})+\varphi_2^2(\tau,\hat{x})]\;d\tau,\\
\pi_s(T_\sigma\varphi(t,\hat{x})) & =E_\alpha(-2 t^\alpha)\sigma_1+\int_0^t(t-\tau)^{\alpha-1}E_{\alpha,\alpha}(-2(t-\tau)^\alpha)\varphi_1^2(\tau,\hat{x})\;d\tau,
\end{align*}
where $B(-\tau)$ is defined in \cite[Lemma 4, pp.~2439--2440]{Deshpande}, and in this case $B(-\tau)=\frac{1}{\alpha}\exp(-2^{\frac{1}{\alpha}}\tau)$. Thus,
\begin{align*}
\pi_u(T_\sigma\varphi(t,\hat{x})) = & E_\alpha(2 t^\alpha)\Big( \int_0^t\frac{(t-\tau)^{\alpha-1}E_{\alpha,\alpha}(2(t-\tau)^\alpha)}{E_\alpha(2 t^\alpha)}[\varphi_1^2(\tau,\hat{x})+\varphi_2^2(\tau,\hat{x})]\;d\tau\\
& -\int_0^\infty \exp(-2^{\frac{1}{\alpha}}\tau)[\varphi_1^2(\tau,\hat{x})+\varphi_2^2(\tau,\hat{x})]\;d\tau\Big).
\end{align*}

Now, if the function $\varphi(t,\hat{x})$ tends to 0 as $t\rightarrow\infty$ as claimed by the authors of \cite{Deshpande} the function $\varphi_1^2(t,\hat{x})+\varphi_2^2(t,\hat{x})$ is bounded, hence by Cong et al.~\cite[Lemma 8, p.~164]{Cong_2014}, we must have
\begin{align*}
\lim_{t\to \infty} & \int_0^t\frac{(t-\tau)^{\alpha-1}E_{\alpha,\alpha}(2(t-\tau)^\alpha)}{E_\alpha(2 t^\alpha)}[\varphi_1^2(\tau,\hat{x})+\varphi_2^2(\tau,\hat{x})]\;d\tau\\
&\hspace*{2cm}= \; 2^{\frac{1}{\alpha}-1}\int_0^\infty \exp(-2^{\frac{1}{\alpha}}\tau)[\varphi_1^2(\tau,\hat{x})+\varphi_2^2(\tau,\hat{x})]\;d\tau.
\end{align*}
Taking into account that $\hat{x}\not=0$, this implies
\begin{align*}
\lim_{t\to\infty}\pi_u(T_\sigma\varphi(t,\hat{x}))=&\\
&\hspace*{-3cm}=\; \Big((2^{\frac{1}{\alpha}-1}-1)\int_0^\infty \exp(-2^{\frac{1}{\alpha}}\tau)[\varphi_1^2(\tau,\hat{x})+\varphi_2^2(\tau,\hat{x})]\;d\tau\Big)
  \lim_{t\to \infty}E_\alpha(2 t^\alpha)\\
& \hspace*{-3cm}=\;\infty,
\end{align*}
what is a contradiction to the claim that $\varphi(t,\hat{x})$ tends to 0 as $t\rightarrow\infty$.

In our opinion, the flaws of the paper \cite{Deshpande} come from the fact that the authors of \cite{Deshpande} used a false representation of the Mittag-Leffler function $E_{\alpha,\beta}(t^\alpha J)$ as they claim the following equality (see \cite[Lemma 4(2), p.~2440]{Deshpande}):
\begin{equation}\label{Mittag-Leffler}
E_{\alpha,\alpha}(t^\alpha J)=t^{1-\alpha}B(t)+C(t),
\end{equation}
where $J$ is the Jordan block corresponding to a complex number $\lambda$ with $|\arg(\lambda)|<\frac{\alpha\pi}{2}$, $B(t)$ and $C(t)$ are defined in \cite[Lemma 4, p.~ 2440]{Deshpande}. To show that the representation they claimed is not true let us consider a matrix $J$ as below:
 \[
 J=   \left(
      \begin{array}{*5{c}}
       \lambda    &     1 \\
     0      & \lambda   \\
      \end{array}
    \right),
\] 
here $0 < \lambda \in \R$, hence $|\arg(\lambda)|=0 <\frac{\alpha\pi}{2}$. Since
\[
E_{\alpha,\alpha}(t^p J)=   \left(
      \begin{array}{*5{c}}
       E_{\alpha,\alpha}(\lambda t^\alpha)    &     \frac{d}{d\lambda} E_{\alpha,\alpha}(\lambda t^\alpha) \\
     0      & E_{\alpha,\alpha}(\lambda t^\alpha)   \\
      \end{array}
    \right),
\] 
by using the integral formula of the Mittag-Leffler function $E_{\alpha,\alpha}(\lambda t^\alpha)$ and the estimates as in Podlubny~\cite[Theorem 1.3, p.~32]{Podlubny}, we can write $E_{\alpha,\alpha}(t^\alpha J)$ in the form:
\begin{equation}\label{assert1}
      \begin{pmatrix}
       \frac{1}{\alpha}\lambda^{\frac{1-\alpha}{\alpha}}t^{1-\alpha}\exp(\lambda^{\frac{1}{\alpha}}t)    &     \frac{1}{\alpha^2}\lambda^{\frac{1-2\alpha}{\alpha}}t^{1-\alpha}\exp(\lambda^{\frac{1}{\alpha}}t)(1-\alpha+t\lambda^{\frac{1}{\alpha}}) \\[10pt]
     0      & \frac{1}{\alpha}\lambda^{\frac{1-\alpha}{\alpha}}t^{1-\alpha}\exp(\lambda^{\frac{1}{\alpha}}t)\\[1.5ex]
      \end{pmatrix}
+C(t),
\end{equation}
here $C(t)$ is the  same as the matrix $C(t)$ defined in \cite[Lemma 4]{Deshpande}.
On the other hand, we can now compute the matrix $B(t)$ in \cite[Lemma 4]{Deshpande} as
\[
B(t)=   \left(
      \begin{array}{*5{c}}
       \frac{1}{\alpha}\exp(\lambda^{\frac{1}{\alpha}}t)    &     \frac{1}{\alpha^2}\lambda^{\frac{1-\alpha}{\alpha}}t\exp(\lambda^{\frac{1}{\alpha}}t) \\
     0      & \exp(\lambda^{\frac{1}{\alpha}}t)   \\
      \end{array}
    \right),
\] 
this together with \eqref{assert1} imply that 
\[
E_{\alpha,\alpha}(t^\alpha J)\neq t^{1-\alpha}B(t)+C(t).
\]
Thus the claim \eqref{Mittag-Leffler} of \cite{Deshpande} is false. This leads to false assertions in Lemma~4(2) and Lemma~8(2) of \cite{Deshpande}, which are main tools in the proof of the stable manifold theorem in \cite{Deshpande}.
Note that if the Jordan block $J$ has higher dimension then the relation between $E_{\alpha,\alpha}(t^\alpha J)$, $B(t)$ and $C(t)$ is even more complicated and one cannot expect a simple relation similar to (the false) equality \eqref{Mittag-Leffler}. As a consequence, the approach to the problem of stable manifolds in  A.\ Deshpande and V.\ Daftardar-Gejji~\cite{Deshpande} is not sound.

\section*{Acknowledgement}
This research of the first, the second and the fourth author is funded by the Vietnam National Foundation for
Science and Technology Development (NAFOSTED) under Grant Number 101.03-2014.42.

\end{document}